\RequirePackage{fix-cm}
\documentclass[smallextended]{svjour3}       
\smartqed  
\usepackage{amsmath,graphicx,amsfonts,amscd,bezier}
\usepackage{makeidx}
\usepackage[latin1]{inputenc}
\journalname{Constructive Approximation}
\begin{document}

\title{Integral identities derived from the complex Funk-Hecke formula
}


\author{C. P. Oliveira \and Jorge Buescu }


\institute{Universidade Federal de Itajub\'{a} - Instituto de Matem\'{a}tica e Computa\c{c}\~{a}o.\\ Caixa Postal 50, 37500-903. Itajub\'{a} - MG, Brasil.\\ \email{oliveira@unifei.edu.br} \\
Universidade de Lisboa - Dep. Matematica and CMAF.\\
\email{jbuescu@ptmat.fc.ul.pt}}

\date{Received: date / Accepted: date}

\maketitle

\begin{abstract}
In this paper we derive integral identities involving both the unit sphere and the unit disk or subsets thereof. In addition these identities lead to a prototype of the Funk-Hecke formula for subspheres embedded in $\Omega_{2q}$. The technique requires the use of the complex Funk-Hecke formula, where eigenvalues of the integral operator generated by a bizonal kernel on the unit sphere $\Omega_{2q}$ of $\mathbb{C}^q$ are given by an integral on the closed unit disk $B_q$ of $\mathbb{C}^{q-1}$.
\end{abstract}

\section{Introduction} \label{intro}

Let $q$ be  a positive integer. We consider the unitary space $\mathbb{C}^q$ equipped with the usual inner product
given by
\begin{eqnarray*}\label{1.1}\langle z,w \rangle := z_1\overline{w}_1 + z_2\overline{w}_2+\cdots +
z_q\overline{w}_q,\quad z,w \in \mathbb{C}^q,\end{eqnarray*} where we write $z=(z_1,z_2,\ldots,z_q)$ and $w=(w_1,w_2,\ldots,w_q)$.

Throughout the paper,  $|\cdot|$ denotes the usual absolute value in $\mathbb{C}$, while $\|\cdot \|$ stands for the norm induced by the inner product $\langle \cdot, \cdot\rangle$.

The unit disk and the unit sphere of $(\mathbb{C}^q, \langle\cdot, \cdot \rangle)$, both centered at the origin, will appear frequently in the text. These are defined respectively by
\begin{eqnarray*}\label{1.2}B_{q+1}:=\{z \in \mathbb{C}^q\,:\, \langle z, z \rangle \leq 1\},\quad \Omega_{2q} := \{z\in\mathbb{C}^q : \langle z,z\rangle
=1\}. \end{eqnarray*}
Observe in particular that, with this notation, $B_2$ stands for the unit disk in $\mathbb{C}$, whereas $ \Omega_{2}$
stands for the unit circle in $\mathbb{C}$.

We will deal with spaces of polynomials in the variables $z$ and $\overline{z}$ of $\mathbb{C}^q$, where $\overline{z}$ denotes complex conjugation. The most general  is $\mathbb{P}(\mathbb{C}^q)$, the space of polynomials in the independent variables
$z$ and $\overline{z}$. Thus, when $\alpha$ and $\beta$ are multi-indices with nonnegative integer coordinates, an element $p$ in $\mathbb{P}(\mathbb{C}^q)$ may be expressed as
\begin{eqnarray*}\label{1.3}p(z):=\sum_{|\alpha|\leq m}\sum_{|\beta|\leq n}
c_{\alpha,\beta}z^\alpha\overline{z}^\beta, \quad c_{\alpha,\beta}\in\mathbb{C}, \quad \alpha,\beta \in
\mathbb{Z}^q_+, \end{eqnarray*} where $m$ and $n$ are nonnegative integers.
We denote the subspace of $\mathbb{P}(\mathbb{C}^q)$ formed by homogeneous polynomials of degree $m$ in $z$ and degree $n$ in $\overline{z}$ by $\mathbb{P}_{m,n}(\mathbb{C}^q)$. Hence,
\begin{eqnarray*}\label{1.4}p(\lambda z)=\lambda^m \overline{\lambda}^{\,n} p(z), \quad \lambda \in \mathbb{C},\quad p \in \mathbb{P}_{m,n}(\mathbb{C}^q). \end{eqnarray*}
The solid harmonics are given by the subspace $\mathbb{H}_{m,n}(\mathbb{C}^q)$ of $\mathbb{P}_{m,n}(\mathbb{C}^q)$. They are formed by the polynomials belonging to the kernel of the complex Laplace operator
\begin{eqnarray*} \label{eqn1.5}{\Delta}_{(2q)}:=4
\sum_{j=1}^{q}\frac{\partial^2}{\partial z_{j}\partial
\overline{z}_j}. \end{eqnarray*} Equivalently,
\begin{eqnarray*} \label{eqn1.6}\mathbb{H}_{m,n}(\mathbb{C}^q) := \mathbb{P}_{m,n}(\mathbb{C}^q) \cap \mbox{Ker} \Delta_{(2q)}.\end{eqnarray*}
The last unitary space that we introduce is $\mathcal{H}_{m,n}(\Omega_{2q})$. This is the space of spherical harmonics on $\Omega_{2q}$.
A complex spherical harmonic is the restriction to $\Omega_{2q}$ of an element of $\cup_{m=0}^{\infty}\cup_{n=0}^{\infty} \mathbb{H}_{m,n}(\mathbb{C}^q)$. We will write $\mathcal{H}_{m,n}(\Omega_{2q})$ to denote the space of complex spherical harmonics that are restrictions  to $\Omega_{2q}$
of elements of $\mathbb{H}_{m,n}(\mathbb{C}^q)$.
The dimension of this space is the finite number $d(m,n)$ (\cite{koorn}). When necessary
\begin{eqnarray*} \label{eqn1.7}\left\{Y^j_{m,n} \, : \, j=1, \ldots, d(m,n)\right\} \end{eqnarray*}
will denote an orthonormal basis of $\mathcal{H}_{m,n}(\Omega_{2q})$ with respect to the usual
 inner product of $L^2(\Omega_{2q}, d\sigma_q)$ given by
\begin{eqnarray*}\langle f, g \rangle_2 = \int_{\Omega_{2q}} f(z) \overline{g(z)}d\sigma_q(z), \quad f, g \in L^2(\Omega_{2q},d\sigma_q).\label{eqno(1.8)}\end{eqnarray*}
The Borel measure $\sigma_q$ on $\Omega_{2q}$ is normalized by the condition
\begin{eqnarray}\omega_q:=\int_{\Omega_{2q}} d\sigma_q(z) = \frac{2\pi^q}{(q-1)!}. \label{eqno(1.9)}\end{eqnarray}
Thus, the constant $\omega_q$ is the ordinary measure of the sphere $\Omega_{2q}$. We recall that $\sigma$ is invariant with respect to isometric operators on $\mathbb{R}^{2q}$ fixing the origin. It is useful to recall that the Lebesgue measure $dz$ on $B_{q +1}$ is such that
\begin{eqnarray*}\int_{B_{q +1}} dz = \frac{\omega_q}{2q} = \frac{\pi^q}{q!}. \label{eqno(1.10)}\end{eqnarray*}
For more details on complex spherical analysis we recommend references \cite{Boyd}, \cite{koorn} and \cite{Koorn2}.

This paper is organized as follows. We reserve Section 2 to address the technical results of harmonic analysis that will be needed for the proof of the main results. In Section 3, we present some comments on the historical origin of the Funk-Hecke formula, a proof of which in our context is also supplied. In Section 4, we use coe\-fficients of the Funk-Hecke formula given by an integral on the closed unit disk $B_q$ obtained in the Section 3 to find integral identities, some of which are known while others are entirely new. As a result, we derive a version of the Funk-Hecke formula for subspheres embedded in $\Omega_{2q}$.

\section{Spherical harmonic analysis}
\label{sec:1}

In this section we review some basic results from spherical analysis. We refer the reader to  references (\cite{Ikeda}, \cite{Ikeda2}, \cite{Ikeda3}, \cite{IkedaKayama}, \cite{koorn}, \cite{Koorn2}, \cite{szego}) for more thorough information on this subject.

We begin with the following lemma which lists some basic properties of
 disk polynomials.  Recall that disk polynomials are only defined for $q> 1$, so whenever they appear
we are always assuming this hypothesis. The disk polynomial of degree $m$ in $z$ and of degree $n$ in $\overline{z}$ associated to the integer $q-2$ is the
function $R_{m,n}^{q-2}$ defined by
\begin{eqnarray*}R_{m,n}^{q-2}(re^{i\theta}):= r^{|m-n|}
e^{i(m-n)\theta}P_{m\wedge n}^{(q-2,|m-n|)}(2r^2-1), \quad re^{i\theta} \in B_2, \label{eqno(2.1)}\end{eqnarray*}
where $m\wedge n := \min \{m,n\}$ and $P_k^{(q-2,m-n)}$ is the usual Jacobi polynomial of degree $k$ associated to the integers $q-2$, $m-n$ normalized by $P_k^{(q-2,m-n)}(1)=1$, $k=0, 1, \ldots$. As a consequence, disk polynomials satisfy the normalization condition $R_{m,n}^{q-2}(1)=1$, $m,n=0,1, \ldots$.

The following properties are now immediate from the above definition.

\begin{lemma} \label{lem(2.1)} Let $q>1$ be an integer and $m$, $n$ nonnegative integers. Then the following properties hold:\\
{\bf 1.} $R_{m,n}^{q-2}(e^{i\theta}z) =
 e^{i(m-n)\theta}R_{m,n}^{q-2}(z),\,\, \theta \in[0,2\pi),\,\,
 z\in B_2 $;\\
{\bf 2.} $|R_{m,n}^{q-2}(z)| \leq 1$, $z\in B_2$;\\
{\bf 3.} $R_{m,n}^{q-2}(\overline{z}) =
 \overline{R_{m,n}^{q-2}(z)} = R_{n,m}^{q-2}(z),\,\, z\in
 B_2$.\end{lemma}

Disk polynomials are connected to spherical harmonics in the following way (\cite{koorn}).

\begin{lemma}[Addition Formula]Let $q>1$ be an integer. If $m$ and $n$ are nonnegative integers, then
\begin{eqnarray}
\label{eq_addition_formula}
R_{m,n}^{q-2}(\langle z, w\rangle) = \frac{\omega_q}{d(m,n)} \sum_{j=1}^{d(m,n)} Y_{m,n}^j(z)\overline{Y_{m,n}^j(w)},\quad z, w \in \Omega_{2q}.
\end{eqnarray}
\end{lemma}

Next let $w$ be in $\Omega_{2q}$ and consider $\gamma$ in $B_2$. The subsphere $\Omega_w^\gamma$ of $\Omega_{2q}$ with radius $(1-|\gamma|^2)^{1/2}$ and pole  $w$ is the intersection of $\Omega_{2q}$ with the hyperplane $\langle z, w \rangle = \gamma$, $z \in \Omega_{2q}$. In other words,
\begin{eqnarray}
\label{def_subsphere}
\Omega_w^\gamma:=\left\{z \in \Omega_{2q}\,:\, \langle z, w \rangle = \gamma\right\}.
\end{eqnarray}

Since the  above definition  makes sense only when $q>1$, this will be a standing hypothesis when these subspheres appear in context.

We remark that $\Omega_w^0$ is a copy of $\Omega_{2q-2}$ embedded in $\Omega_{2q}$ that is orthogonal to $w$ while $\Omega_w^1=\{w\}$. The restriction $\sigma_w^\gamma$ of $\sigma_q$ to $\Omega_w^\gamma$ makes this set a measurable space and we have the following relation between surface elements:
\begin{eqnarray}d\sigma_w^\gamma = (1 - |\gamma|^2)^{(2q-3)/2}d\sigma_{q-1},\quad \gamma \in B_2. \label{eqno(2.4)}\end{eqnarray}

The following technical lemma shows that elements of $\Omega_{2q}$ may be expressed in terms of the pole $w$ and of an element of $\Omega_w^\gamma$.

\begin{lemma}Let $q>1$ be an integer and $w$ be an element of $\Omega_{2q}$. If $z \in \Omega_{2q}$ and $\gamma \in B_2$ such that $|\gamma| \neq 1$, then there exists $te^{i\theta} \in B_2$ and satisfying
\begin{eqnarray} z=\left(t\,e^{i\theta} - \gamma\frac{\sqrt{1-t^2}}{\sqrt{1-|\gamma|^2}}\right)w + \frac{\sqrt{1-t^2}}{\sqrt{1-|\gamma|^2}}z',\quad z' \in \Omega_w^\gamma.\label{eno(2.5)}\end{eqnarray}\end{lemma}

\begin{proof} Let $z \in \Omega_{2q}$ and $\gamma \in B_2$. If $z=\pm w$ then expression \eqref{eno(2.5)} holds for
$te^{i\theta}=\pm 1$ and $z'$ an arbitrary element in $\Omega_w^\gamma$, $|\gamma| \neq 1$.
For the case $z \in \Omega_w^\gamma$ or $z \in \Omega_{-w}^\gamma$ when $|\gamma| \neq 1$, it is enough to consider
$z'=z$ and $te^{i\theta}=\gamma$ or $te^{i\theta}=-\gamma$. Now, assume that $z$ is distinct from $\pm w$ and is not  in
$\Omega_w^\gamma$, $|\gamma| \neq 1$. Spherical projection of $z$ then gives rise to a vector $w' \in \Omega_w^0$ such that
$z=\langle z, w\rangle w + \langle z, w'\rangle w'$. Thus there exist $t,t' \in [0,1]$ and $\theta, \theta' \in [0, 2\pi)$
such that $\langle z, w\rangle=t e^{i\theta}$ and $ \langle z, w'\rangle= t' e^{i\theta'}$. Since $\langle w', w \rangle = 0$ and $\langle z,z \rangle =1$, it follows that
\begin{eqnarray} \label{eqn(26)}z=t e^{i\theta} w + \sqrt{1-t^2}\,e^{i\theta'} w'.\end{eqnarray}
Now consider the vector \begin{eqnarray}\label{eqn(2.7)}z'= \gamma w + \sqrt{1-|\gamma|^2}\, e^{i\theta'} w'\end{eqnarray} in $\Omega_w^\gamma$.
Elimination of $e^{i\theta'}w'$ in \eqref{eqn(26)} and \eqref{eqn(2.7)} leads to the representation \eqref{eno(2.5)},
which concludes the proof. \hfill\rule{2.5mm}{2.5mm} \end{proof}

When $\gamma=0$, expression \eqref{eno(2.5)} leads to the well-known relationship between the measures $d\sigma_q$ and $d\sigma_{q-1}$
\begin{eqnarray}d\sigma_q(z)= t(1-t^2)^{q-2}dt\, d\theta\,d\sigma_{q-1}(z'),\quad te^{i\theta} \in B_2,\label{eqn(2.8)}\end{eqnarray}
where $z=t\,e^{i\theta}w + \sqrt{1-t^2}\,z'$, $z' \in \Omega_{2q-2}$.

Performing a change of variables to polar coordinates we find that (\cite[p.13]{rudin})
\begin{eqnarray}\int_{B_{q+1}}f(\eta)d\eta= \int_0^1\left[\int_{\Omega_{2q}}f(rz)d\sigma_q(z)\right]r^{2q-1}dr, \quad f \in L^1(B_{q+1},d\eta). \label{eqno2.9}\end{eqnarray}

The following theorem is the essential ingredient for our main result. It establishes an identity between the integral over the sphere $\Omega_{2q}$ and the integral over the cylinder $[0,2\pi] \times B_q$.

\begin{theorem}\label{teo(2.1)}
Let $w$ be an element of $\Omega_{2q}$ for  $q>1$. If $f$ is a function of $L^1(\Omega_{2q},d\sigma_q)$, then
\begin{eqnarray*}
\int_{\Omega_{2q}}f(z)d\sigma_q(z) = \int_0^{2\pi}\int_{B_q}f\left(\eta +e^{i\theta}\sqrt{1-\|\eta\|^2}\, w\right)d\eta d\theta.
\label{eqno2.10}
\end{eqnarray*}\end{theorem}

\begin{proof} We consider a $\sigma_q$-integrable function $f$ over $\Omega_{2q}$. Then expression \eqref{eqn(2.8)} implies the equality
$$\int_{\Omega_{2q}}\hspace*{-3mm}f(z)d\sigma_q(z)
 =   \int_0^{2\pi}\int_0^1\left[\int_{\Omega_{2q-2}}\widetilde{f}(z',t,\theta)\,d\sigma_{q-1}(z')\right]t(1-t^2)^{q-2}dt\, d\theta,$$
in which $\widetilde{f}(z',t,\theta)=f\left(te^{i\theta} w+\sqrt{1-t^2} z'\right)$.
Performing the change of variable $r=\sqrt{1-t^2}$ leads to
$$\int_{\Omega_{2q}}\hspace*{-2mm}f(z)\,d\sigma_q(z) =  \int_0^{2\pi}\int_0^1\left[\int_{\Omega_{2q-2}}\hspace*{-4mm}f\left(e^{i\theta}\sqrt{1-r^2}\, w+r z'\right)\,d\sigma_{q-1}(z')\right]r^{2q-3} dr d\theta.$$
Identity \eqref{eqno2.9} applied to the inner integral of the previous equation  leads to \eqref{eqno2.10},
concluding the proof.\hfill\rule{2.5mm}{2.5mm}\end{proof}



In particular, when $w$ is the vector $\varepsilon_q=(0, \ldots, 1) \in \Omega_{2q}$ for $q>1$, we see that \eqref{eqno2.10} assumes the form \begin{eqnarray}\int_{\Omega_{2q}}f(z)d\sigma_q(z) = \int_0^{2\pi}\int_{B_q}f\left(\eta,z_qe^{i\theta}\right)d\eta d\theta, \quad \eta=(z_1, \ldots, z_{q-1}).\label{eqno(2.11)}\end{eqnarray}
This equality thus reduces, in this special case, to the identity of Rudin (\cite[Proposition 1.4.7]{rudin}).

The integral identity below is both the key motivation and, in a sense, the prototype of the general results to be derived in this paper.
Suppose the function under the integral sign in Theorem \ref{teo(2.1)}, which we shall denote by $K$, is {\em bizonal} over $\mathbb{C}^q$ -- that is, its argument depends only on the inner product of $\mathbb{C}^q$. Let  $w \in \Omega_{2q}$. It then follows that
\begin{eqnarray*}\label{eqno(10)}\int_{\Omega_{2q}}K\left(\langle z, w \rangle\right)\,d\sigma_q(z) = \int_0^{2\pi}\int_{B_q}K\left(\sqrt{1-\|\eta\|^2}\,e^{i\theta}\right)d\eta d\theta, \, K \in L^1(B_2, dz).\end{eqnarray*}
We note that this identity is invariant with respect to the pole $w$, a property which shall be explored in Proposition \ref{prop(4.2)} below.

Recall that the measure which connects disk polynomials with
the spherical harmonics present in the Funk-Hecke formula is defined by
\begin{eqnarray}d\nu_q(z):= \frac{q-1}{\pi}(1 - x^2-y^2)^{q-2}dx\,dy, \quad z = x + iy \in B_2,\label{eqno(2.12)}\end{eqnarray}
where $q>1$. Thus the relationship between $d\nu_q$ and  ordinary Lebesgue measure $dz$ is
\begin{eqnarray}d\nu_q(z)= \frac{q-1}{\pi}(1 - |z|^2)^{q-2}dz, \quad z \in B_2.\label{eqno(2.13)}\end{eqnarray}

We now state without proof some standard facts necessary for our main result, to be presented in the next section.
We start with a lemma stating a well-known property about density (\cite{Boyd}).

\begin{lemma} \label{lemma(25)}Let $q>1$. The set $\left\{R_{m,n}^{q-2}\,:\, m, n =0,1, \ldots\right\}$ span the space polynomials on $B_2$ that are restrictions of elements of $\mathbb{P}(\mathbb{C})$. This set is a complete orthogonal system of $L^2(B_2, d\nu_q)$.\end{lemma}

The Stone-Weierstrass approximation theorem is the subject of the next result (\cite{folland}).

\begin{lemma} \label{lemma(2.6)}
Every continuous complex function on compact subsets of $\mathbb{C}$ may be uniformly approximated by a polynomial
in the variables $z$ and $\overline{z}$.\end{lemma}

The proof of the lemma below is standard and will also  be omitted. By $L^1(B_2, d\nu_q)$ we denote the space of  $\nu_q$-integrable functions over $B_2$.

\begin{lemma}
\label{lemma(2.7)}
Let $q>1$ and $(f_n)$ be a sequence of $L^1(B_2, d\nu_q)$. If $(f_n)$ has uniform limit $f$ then $f$ is an element of $L^1(B_2, d\nu_q)$ and
\begin{eqnarray*}
\int_{B_2}f(z) d\nu_q(z) = \lim_{n \to \infty} \int_{B_2} f_n(z) d\nu_q(z).
\label{eqno(2.14)}
\end{eqnarray*}
\end{lemma}

There is a version of Lemma \ref{lemma(2.7)} for functions in $L^1(\Omega_{2q}, d\sigma_q)$.

We close the section with the final result that we will use. It is Lusin's Theorem adapted to our context (\cite{folland}).

\begin{lemma}
\label{lemma(2.8)}
Let $q>1$ and $f$ be a function of $L^1(B_2, d\nu_q)$. Then there exists a sequence $(f_n)$ formed by complex continuous functions on $B_2$
such that
\begin{eqnarray*}
\label{eqno(2.15)}
|f_n(z)| \leq \sup_{w \in B_2} |f(w)|, \, n=0,1, \ldots, \quad \lim_{n \to \infty}f_n(z)=f(z),\, \mbox{a.e.}.
\end{eqnarray*}
\end{lemma}

\section{The Funk-Hecke formula via a complex cylinder}
\label{sec:2}

The purpose of this section is to formulate a complex version of the Funk-Hecke formula, where the coefficients are obtained by an integral over the cylinder $[0,2\pi]\times B_q$. We begin with some background.

The formula in question refers to a class of integral operators generated by a  $\nu_q$-integrable kernel $K$.
These operators have spherical harmonics as eigenfunctions and the corresponding eigenvalues  admit  integral
representations involving both $K$ and certain polynomials (\cite{oliveira}, \cite{quinto}). The importance of this
formula within the framework of  Harmonic Analysis is well-known. The intimate relationship of the formula with
integral operators on the unit sphere is explored in depth  in references \cite{axler}, \cite{Boyd}, \cite{groemer} and \cite{muller}.

The classic formula of Funk-Hecke dates originally from 1915 (\cite{funk}, \cite{hecke}). It refers to the unit sphere $S^{q-1}$ in the euclidian space $(\mathbb{R}^q, \cdot)$, where $\cdot$ is the usual euclidian inner product. Its most common form is the following proposition (\cite[p. 11]{Dai}).

\begin{proposition}
\label{pro(3.1)}
Let $q >1$ and consider  an element $K$ of $L^1([-1,1],dt) \cap L^1([-1,1],(1-t^2)^{(q-3)/2}dt)$.
If $Y$ is a standard spherical harmonic of degree $n$ on $S^{q-1}$, then for each $y \in S^{q-1}$
the function $x \in S^{q-1} \mapsto K(x\cdot y) Y(x)$ is $\tau_q$-integrable over $S^{q-1}$ and
\begin{eqnarray*}\int_{S^{q-1}} K(x\cdot y)Y(x)d\tau_q(x)= \lambda^{q}_n (K)Y(y),\quad y \in S^{q-1},
\label{eqno3.1}
\end{eqnarray*}
where
\begin{eqnarray*}
\lambda^{q}_n(K)=\frac{2(\sqrt{\pi})^{q-1}}{\Gamma\left(\frac{q-1}{2}\right)}\int_{-1}^1 K(t)P_n^{(q-2)/2}(t)(1-t^2)^{(q-3)/2}dt.
\label{eqno(3.2)}
\end{eqnarray*}
\end{proposition}

Here, $\tau_q$ is the Borel measure on $S^{q-1}$, $P_n^{(q-2)/2}$ is the normalized Gegenbauer polynomial of degree $n$ associated to $(q-2)/2$ and $\Gamma$ is the usual gamma function.

Xu Yuan proved  in 2000  another version of the formula on $S^{m-1}$(\cite{xu}). Its formulation depends on a group which is finitely generated by a set of reflexive operators on hyperplanes of $\mathbb{R}^m$. The difference with respect to the present  context is the replacement of spherical harmonics by $h$-spherical harmonics, where $h$ is a suitable weight function. Thus Proposition \ref{pro(3.1)} may be seen as a special case of the theory of Xu where $h$ is the constant function $1$. This theory is extensive and shall not be addressed here, since it is not essential to our purposes.

The complex version of the Funk-Hecke formula was first explored in 1984 in \cite{quinto}. Quinto's proof is based on arguments from group theory and Haar measure. In 2005 a proof using independent arguments was presented in \cite[Theorem 3.5]{oliveira}. The formulation is as follows.

\begin{proposition} \label{prop(3.2)} Let $q>1$ and $w$ be an element of $\Omega_{2q}$. If $K \in L^1(B_2,d\nu_q)$ and $Y \in \mathcal{H}_{m,n}(\Omega_{2q})$ then the function $z \in \Omega_{2q} \mapsto K(\langle z,w\rangle) \overline{Y(z)}$ is $\sigma_q$-integrable and \begin{eqnarray}\int_{\Omega_{2q}} K(\langle z, w \rangle)\overline{Y(z)}d\sigma_q(z)= \lambda_{m,n}^{q-2} (K)\overline{Y(w)},\label{eqno(3.3)}\end{eqnarray}
where
\begin{eqnarray}\lambda_{m,n}^{q-2} (K)= \omega_q\int_{B_2} K(z)\overline{R_{m,n}^{q-2}(z)}d\nu_q(z).\label{eqno(3.4)}\end{eqnarray}\end{proposition}

A distinct but equivalent formulation of this result is derived below. Although the proof of this version of the Funk-Hecke formula follows the same structure as
the proof of Proposition \ref{prop(3.2)}, we stress that it does not reduce to a mere adaptation.

\begin{theorem}
\label{teo(3.3)}
Let $q>1$ and $w$ be an element of $\Omega_{2q}$. If $K \in L^1(B_2,d\nu_q)$ and $Y \in \mathcal{H}_{m,n}(\Omega_{2q})$, then the function $z \in \Omega_{2q} \mapsto K(\langle z,w\rangle) \overline{Y(z)}$ is $\sigma_q$-integrable and \begin{eqnarray}\int_{\Omega_{2q}} K(\langle z, w \rangle)\overline{Y(z)}d\sigma_q(z)= \Lambda_{m,n}^{q-2} (K)\overline{Y(w)}, \label{eqno(3.5)}\end{eqnarray}
where
\begin{eqnarray}\Lambda_{m,n}^{q-2} (K)=\int_0^{2\pi}\int_{B_q} K\left(\sqrt{1-\|\eta\|^2}\,e^{i\theta}\right)\overline{R_{m,n}^{q-2}\left(\sqrt{1-\|\eta\|^2}\,e^{i\theta}\right)}\,d\eta\, d\theta.\label{eqno(3.6)}\end{eqnarray}\end{theorem}

\begin{proof} We begin the proof by analyzing the case where $K$ is the polynomial $R^{q-2}_{k,l}$, in which $k$ and $l$ are nonnegative integers. Then the map $z \in \Omega_{2q} \mapsto R^{q-2}_{k,l}(\langle z,w \rangle)\overline{Y(z)}$ is $\sigma_q$-integrable. We write $Y$ in the form
 $$Y=\sum_{j=1}^{d(m,n)}a_jY_{m,n}^j,\quad a_j \in \mathbb{C} $$ and use the addition formula \eqref{eq_addition_formula},
obtaining
  $$\int_{\Omega_{2q}}R^{q-2}_{k,l}(\langle
 z,w\rangle)\overline{Y(z)}d\sigma_q(z) =
 \frac{\omega_q}{d(k,l)}{\delta}_{km}{\delta}_{ln}\sum_{i=1}^{d(k,l)}
 \left(\sum_{j=1}^{d(m,n)}\overline{a}_j
 {\delta}_{ij}\right)\overline{Y_{k,l}^i(w)}.$$ Moreover, by Theorem \ref{teo(2.1)} and again using \eqref{eq_addition_formula},
 \begin{eqnarray*}\Lambda_{m,n}^{q-2}(K)
 & = & \int_{\Omega_{2q}}R_{k,l}^{q-2}(\langle z, w
 \rangle)\overline{R_{m,n}^{q-2}(\langle z, w \rangle)}d\sigma_q(z)\\
 & = & \frac{\omega_q\delta_{km}}{d(k,l)}\frac{\omega_q\delta_{ln}}{d(m,n)}
 \sum_{i=1}^{d(k,l)}\left(\sum_{j=1}^{d(m,n)}
 \delta_{ij}Y_{m,n}^j(w)\right)\overline{Y_{k,l}^i(w)}.\end{eqnarray*}
 Comparing the two expressions above, it is clear that \eqref{eqno(3.5)} holds when $ (k, l) \neq (m, n) $. When $ (k, l) =
  (m, n) $, direct computation by means of \eqref{eq_addition_formula} again shows that
 \begin{eqnarray*}\int_{\Omega_{2q}}R^{q-2}_{k,l}(\langle
 z,w\rangle)\overline{Y(z)}d\sigma_q(z)=\frac{\omega_q}{d(m,n)}\overline{Y(w)}=\Lambda_{m,n}^{q-2}(K)\overline{Y(w)}. \label{eqno(3.8)}\end{eqnarray*}
For the proof of the general case, we  begin by recalling the density result described in Lemma \ref{lemma(25)}. Since the integral operators in \eqref{eqno(3.4)} and \eqref{eqno(3.5)} are bounded, the formula of theorem holds when $K$ is an element in $\mathbb{P}(\mathbb{C})$ restricted to $B_2$. Suppose now that $K$ is a continuous function in $B_2$.
We then use Lemma \ref{lemma(2.6)} to select a sequence $(p_j)$ in $\mathbb{P}(\mathbb{C})$, restricted to $B_2$ such that $ p_j \to K$ uniformly as $j \to \infty$. Continuity of the operators involved and Lemma \ref{lemma(2.7)} imply that
\begin{eqnarray*}\int_{\Omega_{2q}}K(\langle z, w \rangle)\overline{Y(z)}d\sigma_q(z)=\lim_{j \to \infty}\int_{\Omega_{2q}}p_j(\langle z, w \rangle)\overline{Y(z)}d\sigma_q(z)=\Lambda_{m,n}^{q} (K)\overline{Y(w)}.\end{eqnarray*}
Finally, suppose  $K$ is $\nu_q$-integrable over $B_2$. Select a sequence $(K_j)$ of continuous complex functions as in Lemma \ref{lemma(2.8)} such that
\begin{eqnarray*}\lim_{j \to \infty}K_j(z)=K(z), \quad |K_j(z)| \leq \sup_{w \in \Omega_{2q}}|K(w)|,\quad j=0,1,\ldots,\, \mbox{\rm a.e.}.\label{eqno(3.9)}\end{eqnarray*}
Then for each $w$ in $\Omega_{2q}$,
 \begin{eqnarray*}\lim_{j \to \infty} K_j(\langle z, w\rangle) =K(\langle z,w\rangle), \, z \in \Omega_{2q},\quad \mbox{\rm a.e.}.\label{eqno(3.10)}\end{eqnarray*}
Hence, by Lebesgue's dominated convergence theorem,  identity \eqref{eqno(3.5)} is valid for every function
$K \in L^1(B_2, d\nu_q)$. This completes the proof.\hfill\rule{2.5mm}{2.5mm}\end{proof}

Note that the eigenvalues that appear in \eqref{eqno(3.4)} and \eqref{eqno(3.6)} are the same, and therefore
the different integral expressions obtained are equivalent. However, the equality of these two expressions is not trivial
and seems to be unknown in the literature.

Since we are dealing with two complex variables and complex functions, the above formula has several interesting
variants which we present below. These equalities use the fact that the operation of complex conjugation establishes  a bijection from $\mathcal{H}_{m,n}(\Omega_{2q})$  to $\mathcal{H}_{n,m}(\Omega_{2q})$.

\begin{corollary} Under the hypotheses of the above theorem, the following identities hold:
$$\int_{\Omega_{2q}} K(\langle z, w \rangle)Y(z)d\sigma_q(z)= \Lambda_{n,m}^{q-2} (K)Y(w), $$
$$\int_{\Omega_{2q}} K(\langle w, z \rangle)Y(z)d\sigma_q(z)= \Lambda_{m,n}^{q-2} (K)Y(w), $$
$$\int_{\Omega_{2q}} K(\langle w, z \rangle)\overline{Y(z)}d\sigma_q(z)= \Lambda_{n,m}^{q-2} (K)\overline{Y(w)}.$$\end{corollary}

\section{Integrals relating the unit sphere to its subspheres}
\label{sec:3}

In this section we will obtain integral representations relating $L^1([0,1], dt)$, $L^1(B_q, d\eta)$, $L^1(\Omega_{2q},d\sigma_q)$ and $L^1(B_2,d\nu_q)$. This will also allow the construction of a kind of Funk-Hecke theorem for subspheres $\Omega_w^\gamma$.

\begin{proposition} \label{prop(4.1)}Let $q>1$ and $w$ be an element of $\Omega_{2q}$. If $f$ is a function of $L^1(\Omega_{2q},d\sigma_q)$, then
\begin{eqnarray*}\int_{\Omega_{2q}}f(z)d\sigma_q(z) & = & \int_0^{2\pi}\int_{B_q}f\left(\eta +e^{i\theta}\sqrt{1-\|\eta\|^2}\,\varepsilon_q\right)\,d\eta\,d\theta\\ &  =  & \int_0^{2\pi}\int_{B_q}f\left(\eta+e^{i\theta}\sqrt{1-\|\eta\|^2}\, w\right)d\eta d\theta,\end{eqnarray*}
where $\eta=(z_1, \ldots, z_{q-1})$ and $\varepsilon_q=(0,\ldots, 1)$.\end{proposition}

\begin{proof} The first equality follows from \eqref{eqno(2.11)}, while the second is a consequence of Theorem \ref{teo(2.1)}.\hfill\rule{2.5mm}{2.5mm}\end{proof}

The statement of Proposition \ref{prop(4.1)} for the special case of a bizonal function is dealt with  in the following propositions.

\begin{proposition} \label{prop(4.2)}Let $q>1$ and $w$ be an element of $\Omega_{2q}$. If $K \in L^1(B_2,d\nu_q)$, then the function $z \in \Omega_{2q} \mapsto K(\langle z,w\rangle)$ is $\sigma_q$-integrable and
\begin{eqnarray*} \int_{\Omega_{2q}}K(\langle z, w \rangle)\,d\sigma_q(z)  & = & \omega_q\int_{B_2}K(z)d\nu_q(z)\\
& = & \int_0^{2\pi}\int_{B_q} K\left(e^{i\theta}\sqrt{1-\|\eta\|^2}\right)d\eta\,d\theta\\
& = & \omega_{q-1}\int_0^{2\pi}\int_0^1K\left(te^{i\theta}\right)t(1-t^2)^{q-2}dt\,d\theta.\end{eqnarray*}\end{proposition}

\begin{proof} The equalities cames of Theorem \ref{teo(3.3)} and Proposition \ref{prop(3.2)}. \end{proof}

\begin{proposition}\cite[Formula 1.4.3.]{rudin}  \label{prop(4.3)}Let $q>0$ and $K$ be a continuous function over $B_{q+1}$. If $w$ is an element in $B_{q+1}$, then
\begin{eqnarray*}\int_{B_{q+1}}K\left(\langle \eta, w \rangle\right)d\eta= \int_0^1\left[\int_{\Omega_{2q}}K(\langle r z,w \rangle)d\sigma_q(z)\right]r^{2q-1}dr. \label{eqno(4.1)}\end{eqnarray*}\end{proposition}

In particular, when $K$ is a continuous radial function the previous proposition implies that
\begin{eqnarray*}\int_{B_{q+1}}K(\|\eta\|)d\eta= \omega_q\int_0^1K(r)\,r^{2q-1}dr, \quad K(\|\cdot\|) \in L^1(B_2,dz).\label{eqno(4.2)}\end{eqnarray*}

Using the orthogonality relations for disk polynomials (\cite{Boyd}), we see that Proposition \ref{prop(4.2)} also implies
the following orthogonality property of these polynomials over the cylinder $[0, 2\pi] \times B_q$:
\begin{eqnarray*}\int_0^{2\pi}\hspace{-2mm}\int_{B_q}\hspace{-2mm} R_{\mu,\nu}^{q-2}\left(\sqrt{1-\|\eta\|^2}\,e^{i\theta}\right)\overline{R_{m,n}^{q-2}\left(\sqrt{1-\|\eta\|^2}\,e^{i\theta}\right)}\,d\eta\, d\theta
=c(m,n, q)\delta_{\mu m}\delta_{\nu n},\end{eqnarray*}
where $\delta_{\mu m}$ is  Kronecker's delta function and
$$c(m,n,q)=\frac{2\pi^q m!n!(q-2)!}{(m+n+q-1)(m+q-2)!(n+q-2)!}.$$
Then  Lemma \ref{lem(2.1)} implies that
\begin{eqnarray*}\int_{B_q} R_{\mu,\nu}^{q-2}\left(\sqrt{1-\|\eta\|^2}\right)\overline{R_{m,n}^{q-2}\left(\sqrt{1-\|\eta\|^2}\right)}\,d\eta
=\frac{c(m,n, q)}{2\pi}\delta_{\mu -\nu}\delta_{m-n}.\end{eqnarray*}

Our last result is a prototype of the Funk-Hecke formula for subspheres $\Omega_w^\gamma$.
\vspace{3mm}

\begin{theorem} \label{teo(4.3)}Let $q>1$ and consider $w$ in $\Omega_{2q}$.
If $K$ is a complex function over $B_2$ and $Y$ in $\mathcal{H}_{m,n}(\Omega_{2q})$, then
\begin{eqnarray}\int_{\Omega_w^\gamma} K(\langle z, w \rangle)\overline{Y(z)} d\sigma_w^\gamma(z) = \Upsilon_{m,n}^\gamma(K)\overline{Y(w)}, \quad \gamma \in B_2,\label{eno(4.3)} \end{eqnarray}
where
\begin{eqnarray}\Upsilon_{m,n}^\gamma(K) = \omega_{q-1}(1-|\gamma|^2)^{(2q-3)/2}K(\gamma)\overline{R_{m,n}^{q-2}(\gamma)}.\label{eqno(4.4)}\end{eqnarray}\end{theorem}

\begin{proof} We consider $Y \in  \mathcal{H}_{m,n}(\Omega_{2q})$ and $K: B_2 \to \mathbb{C}$ a function. Applying the decomposition \eqref{eqno(2.4)} it follows that for each $\gamma \in B_2$
\begin{eqnarray}\int_{\Omega_w^\gamma}K(\langle z, w \rangle)\overline{Y(z)} d\sigma_w^\gamma(z)  =  (1-|\gamma|^2)^{(2q-3)/2}K(\gamma)\int_{\Omega_{2q-2}}\hspace{-3mm}\widetilde{Y}(z'')d\sigma_{q-1}(z''),\label{eno(4.5)}\end{eqnarray}
where $\widetilde{Y}(z'')=\overline{Y}(\gamma w + \sqrt{1-|\gamma|^2}\,z'')$.
To compute the last integral we use decomposition \eqref{eqn(2.8)}, obtaining
$$\int_{\Omega_{2q}}\overline{Y(z)} d\sigma_q(z) = \int_0^{2\pi}\int_0^1 \left[\int_{\Omega_{2q-2}}\widetilde{Y}(z'') d\sigma_{q-1}(z'')\right] t(1-t^2)^{q-2}dt d\theta,\, w \in \Omega_{2q}.$$
On the other hand, Theorem \ref{teo(3.3)} and the previous proposition imply that
$$ \int_{\Omega_{2q}}\overline{Y(z)} d\sigma_q(z) = \int_0^{2\pi}\hspace{-2mm}\int_0^1\left[\omega_{q-1} \overline{R_{m,n}^{q-2}(te^{i\theta})}\overline{Y(w)}\right]t(1-t^2)^{q-2}dtd\theta.
$$
Comparing these integrals we see that
$$\int_{\Omega_{2q-2}}\widetilde{Y}(z'')d\sigma_{q-1}(z'')=\omega_{q-1} \overline{R_{m,n}^{q-2}(te^{i\theta)}}\overline{Y(w)}.$$
Inserting the last integral  in \eqref{eno(4.5)} we obtain formula \eqref{eno(4.3)}, concluding the proof. \end{proof}

\noindent{\bf Remark.} Apart from being defined on $B_2$, no hypothesis on $K$ is necessary for the validity of \eqref{eno(4.3)} and \eqref{eqno(4.4)}. This holds because $\langle z, w \rangle = \gamma$ is constant over $\Omega_w^\gamma$ where integration is being performed. Thus the presence of  $K$ in the statement of Theorem \ref{teo(4.3)} is not strictly necessary; however, we chose to state it in these terms so that it acts as a prototype of the Funk-Hecke formula for subspheres.
\vspace{2mm}

We next point out some results which are immediately derived from the above theorem. It is noteworthy to observe that
\begin{eqnarray}\sigma_w^\gamma(\Omega_w^\gamma)= \frac{2\pi^{q-1}}{(q-2)!} (1 - |\gamma|^2)^{(2q-3)/2}, \quad \gamma \in B_2,\, w \in \Omega_{2q}. \label{eqno(4.6)}\end{eqnarray}
In particular, when $\gamma=0$ the constant above is the measure of the sphere $\Omega_{2q-2}$, which is given by \eqref{eqno(1.9)}. For the case $\gamma=1$, we have $\sigma_w^1(\Omega_w^1)=0$, which coincides with the measure of the set $\Omega_w^1=\{w\}$.

In addition, Theorem \ref{teo(4.3)} provides an interesting geometric interpretation of the Funk-Hecke formula. The mean value property for integrals of $Y \in \mathcal{H}_{m,n}(\Omega_{2q})$ over the subsphere $\Omega_w^\gamma$ may be expressed by
\begin{eqnarray*}\int_{\Omega_w^\gamma} Y(z)\, d\sigma_w^\gamma(z) = \sigma_w^\gamma(\Omega_w^\gamma)R_{m,n}^{q-2}(\gamma) Y(w), \quad \gamma \in B_2, \, w \in \Omega_{2q}.\label{eqno(4.7)}\end{eqnarray*}

The mean value property appears in many problems associated to shift operators or to convolution operators over $\Omega_w^\gamma$.

The  analog of Theorem \ref{teo(4.3)} when the argument of $K$ is now allowed to depend on a fixed point $w' \in \Omega_{2q}$ distinct from $w$  remains an open and interesting problem.

\begin{acknowledgements}
The first author was supported by CAPES - BEX 10884/13-0.\\ The secund author was partially supported by Funda\c{c}\~{a}o para a Ciencia e Tecnologia,
PEst-OE/MAT/UI0209/2011.
\end{acknowledgements}



\end{document}